\documentclass[a4paper,10pt]{amsart}
\usepackage[english]{babel}
\usepackage[a4paper, margin=2cm]{geometry}
\usepackage{graphicx}
\usepackage{amsmath, amsthm, amssymb}
\usepackage{enumitem}
\setlist[1,enumerate]{label={(\roman*)}}
\setlist[1]{leftmargin=1.5em}
\setlist[2,enumerate]{label={(\alph*)}}
\setlist{noitemsep, topsep=\smallskipamount, listparindent=1em}
\usepackage{hyperref}
\hypersetup{
    colorlinks=true,
    citecolor=blue
}
\usepackage{tikz-cd}

\def\floor#1{\left\lfloor#1\right\rfloor}
\def\ceil#1{\left\lceil#1\right\rceil}

\let\kve\Q
\let\er\R
\let\ce\C
\let\pe\P
\def\O{\mathcal O}
\def\oct{\mathbb O}

\def\reg{\text{reg}}

\let\phi\varphi

\def\zav#1{\left(#1\right)}
\def\set#1{\left\{#1\right\}}
\def\gener#1{\left\langle#1\right\rangle}
\def\mtrx#1{\begin{pmatrix}#1\end{pmatrix}}
\let\surjto\twoheadrightarrow
\let\injto\hookrightarrow

\DeclareMathOperator{\Hom}{Hom}
\DeclareMathOperator{\rank}{rank}
\DeclareMathOperator{\im}{im}
\DeclareMathOperator{\codim}{codim}
\DeclareMathOperator{\RND}{RND}
\DeclareMathOperator{\ND}{ND}
\DeclareMathOperator{\Seg}{Seg}
\DeclareMathOperator{\GR}{GR}
\DeclareMathOperator{\GL}{GL}
\DeclareMathOperator{\SL}{SL}
\DeclareMathOperator{\pr}{pr}

\DeclareMathOperator{\sgn}{sgn}

\def\uv#1{``#1''}

\newtheorem{theorem}{Theorem}

\newtheorem{lemma}[theorem]{Lemma}
\newtheorem{corollary}[theorem]{Corollary}
\theoremstyle{definition}
\newtheorem{definition}[theorem]{Definition}
\newtheorem{example}[theorem]{Example}
\theoremstyle{remark}
\newtheorem*{remark}{Remark}

\title{On degenerate geometric rank}

\author{Matěj Doležálek}
\address[Matěj Doležálek]{Fachbereich Mathematik und Statistik, Universität Konstanz, Konstanz, Germany}
\email{matej.dolezalek@uni-konstanz.de}

\author{Paweł Pielasa}
\address[Paweł Pielasa]{Department of Pure Mathematics and Mathematical Statistics, University of Cambridge, Wilberforce Road Cambridge CB3 0WB, United Kingdom}
\email{pp554@cam.ac.uk}
\curraddr{Department of Mathematics, Princeton University, Washington Rd, Princeton, NJ 08540, United States}
\email{pp2629@princeton.edu}

\author{Derek Wu}
\address[Derek Wu]{Texas A\&M University, College Station, Texas, USA}
\email{dwu120@tamu.edu}


\begin{document}

\begin{abstract}
    A tensor of degenerate geometric rank corresponds to a linear space of matrices in which some locus of matrices of rank at most $r$ has unexpectedly large dimension. We prove a sufficient criterion for when such a~tensor cannot be lifted to a larger tensor and characterize those tensors for which the degeneracy is caused by the rank $r=1$ locus. We also provide new examples of tensors with degenerate geometric rank stemming from nonlinear loci.  
\end{abstract}

\maketitle

\section{Introduction}

Tensors play an increasingly large role in mathematics and computer science.
Various notions of \emph{rank} of tensors have been developed for different applications; we refer the reader to \cite{landsberg2011tensors} for a broad introduction. Kopparty, Moshkovitz and Zuiddam recently introduced \emph{geometric rank} \cite{kopparty-moshkovitz-zuiddam}, motivated by algebraic complexity theory, extremal combinatorics, and quantum information theory. We state its definition in Section~\ref{GR}; informally, it measures unexpectedly large dimensions of rank loci within the corresponding space of matrices.
Like many other notions of rank, geometric rank generalizes one of the definitions of matrix rank: namely, the fact that matrix rank is the codimension of both the left and right null spaces. Geometric rank also relates closely to \emph{slice rank} and \emph{analytic rank} \cite{cohen-moshkovitz}, notions motivated respectively by combinatorial and statistical questions. 

Given a tensor $T\in \ce^m\otimes \ce^m\otimes \ce^m$, we denote its geometric rank by $\GR(T)$.
We always have $\GR(T)\leq m$, and if $\GR(T)<m$, we say $T$ is of \emph{degenerate geometric rank}.
The study of tensors of degenerate geometric rank, initiated in \cite{geng-landsberg} and continued in \cite{geng}, is useful for the above mentioned areas, since geometric rank is an upper bound for \emph{border subrank}, which measures a tensor's computational ``value''.
Additionally, being of bounded geometric rank is a Zariski closed condition, so the varieties of tensors of degenerate geometric rank may be of interest in their own right.

In the study of geometric rank, it is natural to identify tensors with linear spaces of matrices. In this setting, spaces of matrices of bounded rank are a particular way of achieving degenerate geometric rank, at least for balanced tensors. Such spaces are a classical object of study, dating back to at least Dieudonn\'e in 1949 \cite{dieudonne} and Flanders in 1962 \cite{flanders}. 
Classification results are known up to rank four \cite{atkinson1980large, atkinson1981primitive, atkinson1983primitive, huang-landsberg}. Such classification efforts involve the notion of \emph{liftability}, or whether a space of bounded rank is just a subspace of a larger space of bounded rank. Unliftability of a space may often be verified by a sufficient criterion due to Draisma \cite{draisma}. In principle though, tensors of degenerate geometric rank ought to be a much broader class than spaces of bounded rank, hence one of our interests in this article is to explore examples of degenerate geometric rank where the degeneracy does not stem from a linear space.

Within the tensor geometric community, a particular interest is usually paid to \emph{matrix multiplication} tensors $M_{\gener n}$ on account of their importance to complexity theory. Indeed, the original introduction of geometric rank in \cite{kopparty-moshkovitz-zuiddam} used the new notion to upper-bound the border subrank of matrix multiplication tensors with $\GR(M_{\gener n}) = \ceil{\frac34 n^2}$, proving equality in a lower bound due to Strassen \cite{strassen-subrank}.
This also particularly shows that the geometric rank of $M_{\gener n}$ is degenerate, and this degeneracy is caused by loci of high degree for large $n$ (see Example~\ref{ex:matmult}). To our knowledge, other tensors with degenerate geometric rank not stemming from a linear subspace of bounded rank have not been identified so far -- we present in the final section several examples of such behavior we were able to construct.

Another motivation for studying degenerate geometric rank relates to tensor flattenings. Well-established methods may be seen as stemming from particular spaces of matrices of bounded rank (cf. \cite[Proposition 8.2.3.1]{landsberg2017geometry}): as an example, the first Koszul flattening $\bigwedge^2 A^* \otimes B^* \to A^*\otimes C$ (see e.g. \cite[Sections 3.10.1, 3.10.2]{landsberg2011tensors} or \cite[Section 1.2]{landsberg-ottaviani}) relates to the space $A^*\injto A\otimes \bigwedge^2 A^* = \Hom\left(A^*, \bigwedge^2 A^*\right)$ of bounded rank $\dim A - 1$. The Koszul flattening then constitutes a linear inclusion of $A\otimes B\otimes C$ into the matrix space $\Hom\left(\bigwedge^2 A^* \otimes B^*, A^*\otimes C\right)$. One may hope then to use more general degenerate rank loci to construct good nonlinear inclusions leading to new nonlinear flattenings (cf. \cite{dolezalek2026nonlinearmethodstensorsdeterminantal}).

\null

Our paper is structured as follows. In Section~\ref{sec:liftability}, we investigate when a degenerate rank locus of a tensor {lifts} to a larger tensor. We define the set of \emph{rank neutral directions} and prove that when these coincide with the space of matrices given by the tensor, the degenerate rank locus cannot lift -- see Definition~\ref{def:rndr} and Theorem~\ref{thrm:main-liftability}. This generalizes analogous results of Draisma \cite{draisma} on spaces of matrices of bounded rank. A limited implementation of the criterion in SageMath is available at \cite{code-repo}.

In Section~\ref{sec:rank-one-degeneracies}, we investigate when degenerate geometric rank may be caused by a tensor's rank $1$ locus. This phenomenon may be seen as the opposite extreme to spaces of matrices of bounded rank, where a very large locus, namely the whole space, is the culprit. In Theorem~\ref{thrm:pawel}, we prove that either the corresponding space of matrices contains a linear space of rank $1$ matrices, or a space equivalent to $2\times 2$ symmetric matrices. Since in the latter case, the rank $2$ locus on its own also causes the degenerate geometric rank, the theorem implies that degenerate $\GR(T)$ cannot be caused solely by the rank $1$ locus \uv{in an interesting way}.

Lastly, in Section~\ref{sec:examples} we list examples of nonlinearly degenerate geometric rank that we are aware of.

\section{Preliminaries}

\subsection{Conventions and notation}
We work over $\ce$.
We use both projective and affine spaces, so we consider $\pe^{m-1}$ as some $\pe V$ for an $m$-dimensional vector space $V$.
For any nonzero $v\in V$, we use $[v]$ to denote the point in $\pe V$ corresponding to the line $\ce v$.
For $X\subseteq\pe V$ a projective scheme, we denote by $\hat X\subseteq V$ its affine cone and by $X_\reg$ its regular locus.
We write the embedded tangent space at $x\in X$ as $T_xX$ and the affine tangent cone as $\hat T_x X$.
For a set $S\subset\pe V$, let $\gener S\subset \pe V$ denote its projective span.

Let $A$, $B$, and $C$ all be complex vector spaces of dimension $m$; we use $A^*$ to denote the dual vector space to $A$.
We denote the Segre variety as $\Seg(\pe B\times\pe C)\subseteq  \pe(B\otimes C)= \pe(\Hom(B^*, C))$.
The $r$-th secant variety of a projective variety $X$ is the Zariski closure of the union of projective spans of $r$-tuples of points on $X$, i.e.
\[
    \sigma_r(X) := \overline{\bigcup_{x_1,\dots,x_r\in X}\gener{x_1,\dots,x_r}}.
\]
In particular, the set of matrices of rank at most $r$ in $\Hom(B^*,C)$ is precisely $\hat\sigma_{r}\Seg(\pe B\times\pe C)$.

In general for a tensor $T\in A\otimes B\otimes C$, there are three \emph{flattening maps} we can build from it:
\begin{gather*}
  T_A:A^*\to B\otimes C\quad\quad T_B:B^*\to A\otimes C\quad\quad T_C:C^*\to A\otimes B.
\end{gather*}
We say $T$ is \emph{concise} of all three of those flattening maps are injective.
Throughout this paper, we assume all tensors are concise.

Linear subspaces $E\subseteq B\otimes C = \Hom(B^*, C)$ of dimension $m$ correspond to tensors $T\in A\otimes B\otimes C$.
The correspondence is via identification $E=T_A(A^*)$ up to the $\GL(A)$-action on $A\otimes B\otimes C$.
Through this identification, we may refer to the image of the flattening map, the \emph{space of matrices}, and the tensor itself interchangeably.

\subsection{Geometric rank}\label{GR}

\begin{definition}
Consider a tensor $T\in A\otimes B\otimes C$ and its flattening $T_A:A^*\to B\otimes C$.
We define the \emph{rank $r$ locus of $T$} as
\[
    Y^r_T := \set{[\alpha]\in\pe A^* \mid \rank T_A(\alpha)\leq r}
    =\pe\Bigl(T_A^{-1}\bigl(T_A(A^*)\cap \hat\sigma_{r}\Seg(\pe B\times\pe C)\bigr)\Bigr)
    \subseteq\pe A^*.
\]
Then, the \emph{geometric rank} of $T$ is the integer
\[
    \GR(T) := \min_{r}\bigl(\codim_{\pe A^*}(Y^r_T)+r\bigr).
\]
We say the integer $r$ or the locus $Y^r_T$ \emph{achieves} $\GR(T)$ if $\GR(T)$ is attained by $\codim_{\pe A^*}(Y^r_T)+r$.
The tensor $T$ has \emph{degenerate geometric rank} if $\GR(T)<\min\set{\dim A,\dim B,\dim C}$.
\end{definition}
This is one of several equivalent ways to define $\GR(T)$; in particular, $\GR(T)$ does not change if we work with one of the other two flattenings instead of $T_A:A^*\to B\otimes C$. See \cite[Sections 2 and 3]{kopparty-moshkovitz-zuiddam} or \cite[Proposition/Definition~2.4]{geng-landsberg} for further details.

Note that the indexing of sets $Y^r_T$ is reversed compared to the sets $\Sigma^A_j$ used in \cite{geng-landsberg} and the notation for the \emph{$(A,j)$-geometric rank} $\GR_{A,j}$. The translation between these is as follows:
\begin{align*}
    \Sigma^A_j &= Y^{\min\set{\dim B,\dim C}-j}_T,\\
    \GR_{A,j} &= \dim A+\min\set{\dim B,\dim C}-1-\dim\Sigma^A_j-j \\&= \codim_{\pe A^*}\zav{Y^{\min\set{\dim B,\dim C}-j}_T}+(\min\set{\dim B,\dim C}-j).
\end{align*}

\begin{remark}
    The rank loci are related to the \emph{collineation varieties} of \cite{gesmundo2026collineation}. Namely, treating $T$ as a matrix whose entries are linear forms on $A^*$, its $(r+1)\times(r+1)$ minors define a linear system of divisors on $\pe A^*$.
    $Y_T^r$ is the base locus of this system, while the $(r+1)$-th collineation variety of $T$ is the closure of the image of the rational map $\pe A^* \dashrightarrow \pe\zav{\bigwedge^{r+1} B \otimes\bigwedge^{r+1} C}$ given by the system.

    Cf. also \cite[Section 6.2]{dolezalek2026nonlinearmethodstensorsdeterminantal} for connections to nonlinear flattenings of tensors.
\end{remark}

\section{Liftability of rank loci}
\label{sec:liftability}

When a tensor $T$ is of degenerate geometric rank, it may come from a restriction of a larger tensor $T'$.
If we wish to classify tensors of degenerate geometric rank, it would be ideal to only consider tensors that we cannot enlarge further.
We establish a sufficiency condition that can check if a tensor is \emph{unliftable}, generalizing the ideas presented by Draisma  in \cite{draisma}.

\begin{definition}
    A space of matrices $E\subseteq \Hom(B^*,C)$ is \emph{$r$-liftable} if it is contained within a larger space $E'\supsetneq E$ such that
    \[
        \codim_{\pe E}(\pe E\cap \sigma_{r}\Seg(\pe B\times\pe C)) = \codim_{\pe E'}({\pe E'}\cap \sigma_{r}\Seg(\pe B\times\pe C)).
    \]
    Then $E'$ is an \emph{$r$-lift of $E$} (or \emph{$E$ $r$-lifts to $E'$}).
    Otherwise, $E$ is \emph{$r$-unliftable}.
\end{definition}

\begin{definition}
    \label{def:rndr}
    Consider a space of matrices $E\subseteq \Hom(B^*,C)$
    and denote by $Y:= \pe E\cap\sigma_{r}\Seg(\pe B\times\pe C)$ the locus of matrices of rank at most $r$, considered as an intersection scheme-theoretically.
    Let $Y_1,\ldots, Y_k$ be the maximum-dimensional irreducible components of $Y$, and for each $i$, let $\tilde Y_i$ be the locus of points $[e]\in Y_i$ which are regular points of $Y$ such that $\rank(e) = r$.
    Then, we define the \emph{set of rank-$r$-neutral directions} of $E$ as
    \[
        \RND_r(E) := \bigcup_{1\leq i\leq k} \bigcap_{[e]\in \tilde Y_i} \zav{E + \set{M\in \Hom(B^*,C)\mid M(\ker e)\subseteq \im e}}\supseteq E.
    \]
    If $\tilde Y_i=\emptyset$, we interpret the corresponding intersection as the entire space $\Hom(B^*,C)$.
\end{definition}

\begin{theorem}
    \label{thrm:main-liftability}
    If $E$ $r$-lifts to an $E'=E+\ce v$, then $v\in \RND_r(E)$. If $E=\RND_r(E)$, then $E$ is $r$-unliftable.
\end{theorem}

The Proposition may be useful either as a criterion of unliftability, or as a tool which hints at a possible lift.
In particular, if $Y$ is irreducible and $\RND_r(E)$ is a linear space of dimension $\dim E+1$, then the only possible lift is $E':=\RND_r(E)$ itself.
In practice, one wishes to only use finitely many $[e]\in\tilde Y_i$ in each of the intersections in the definition of $\RND_r(E)$: any such choice will give a superset of $\RND_r(E)$.
See Section~\ref{sec:examples} for examples of this.

We provide a limited implementation of the above criterion in SageMath at \cite{code-repo}; it only functions under the assumption that $Y=Y_1$ is irreducible and $\tilde Y_1$ is nonempty. It relies on the ability to randomly generate points $[e]\in Y$; this is justified by $\tilde Y_1$ being dense in $Y$. Working over $\kve$ to avoid floating-point arithmetic, it may not always be possible to choose rational points on $Y$, since the set of rational points might be difficult to work with or even empty. We therefore only solve this problem ad hoc for varieties $Y$ that occur in our cases of interest

\null

To prove Theorem~\ref{thrm:main-liftability}, let us generalize the approach of \cite[Section 3]{draisma}. As in there, we generalize by replacing $\sigma_{r} \Seg(\pe B\times\pe C)$ with an arbitrary projective variety. Consider a projective variety $X\subseteq \pe V$ and a~linear subspace $E\subseteq V$ with the property that $\codim_{\pe E}(X \cap \pe E) = m$. Whenever $E'\supseteq E$ is a larger space, $\codim_{\pe E'}(X\cap \pe E')$ will be at least $m$. We wish to provide a criterion for when $E$ is inclusion-maximal among spaces with $\codim_{\pe E}(X\cap \pe E)=m$.

\begin{definition}
    Let us say $E$ is \emph{$X$-liftable} (resp. \emph{$X$-unliftable}) if it is contained in some (resp. is not contained in any) $E'\supsetneq E$ such that
    \[
        \codim_{\pe E'} (X\cap \pe E') = m.
    \]
\end{definition}

\begin{definition}
    \label{def:NDX}
    Let $Y:=X\cap\pe E$ considered as a scheme-theoretic intersection, let $Y_1,\dots,Y_k$ be the maximum-dimensional irreducible components of $Y$ and denote $\tilde Y_i := Y_i\cap Y_\reg$.
    Then we define the set of \emph{$X$-neutral directions of $E$} as
    \[
        \ND_X(E) := \bigcup_{1\leq i\leq k}\ \bigcap_{e\in \tilde Y_i}(E+\hat T_{e}X).
    \]
\end{definition}
\begin{remark}
    \begin{enumerate}
    \item Note that in the simplest case when $Y=X\cap \pe E$ is irreducible, the set of $X$-neutral directions simplifies to
    $\ND_X(E) = \bigcap_{e\in Y_\reg} (E+\hat T_e X)$ and is a linear subspace of $V$.
    \item If a $\tilde Y_i$ is empty, we interpret the intersection indexed by it as the whole space $V$.
    \end{enumerate}
\end{remark}

Let us now fix $E$ and define
\[
    U := \set{v\in V\mid \text{$\codim_{\pe E'}(X\cap \pe E')=m$ for $E':= E+\ce v$}},
\]
this is an affine variety in $V$. Note that trivially $E\subseteq U$. On the other hand, we will bound $U$ by $\ND_X(E)$:

\begin{lemma}
    \label{lem:general}
    $U\subseteq \ND_X(E)$.
\end{lemma}
\begin{proof}
    Let us consider an arbitrary $v\in U$ and prove it lies in $\ND_X(E)$. For $v\in E$ this is trivial, so let us presume $v\notin E$ and denote $E':= E+v\ce$. This has $\dim E' = \dim E+1$, and so by the definition of $U$, adopting the notation of Definition~\ref{def:NDX} and setting $d:=\dim Y$, $Z:=X\cap \pe E'$ (scheme-theoretically), we must have $\dim Z = d+1$. Therefore, $Z$ has some irreducible component $Z_1$ of dimension $d+1$.
    Further, $Y=Z\cap \pe E$ and $\pe E$ is a hyperplane in $\pe E'\supset Z$, hence $\dim (Z_1\cap \pe E) = d$. Thus one of the maximal-dimensional components of $Y$ is an irreducible component of $Z_1\cap \pe E$, without loss of generality we may assume it is $Y_1$.

    Consider any $e\in \tilde Y_1 = Y_1\cap Y_\reg$. Points in the intersection of two or more components of $Y$ are singular, hence $e\in Y_1\setminus (Y_2\cup\cdots\cup Y_k)$ and $\hat T_e Y = \hat T_e Y_1$ is $d$-dimensional due to regularity. Further, $e$ lies in $Z_1$, a~$(d+1)$-dimensional component of $Z$, hence
    \[
        \dim \hat T_e Z \geq d+1 > d = \dim \hat T_e Y.
    \]
    Simultaneously, $\hat T_e Y = \hat T_e Z\cap E$ due to the scheme-theoretic intersection $Y=Z\cap \pe E$. The dimension count thus shows that $\hat T_eZ\cap E$ is a proper subspace of $\hat T_eZ$, i.e. $\hat T_eZ\nsubseteq E$.

    Thus, we may choose a vector $w\in \hat T_e Z\setminus E$. Since $E$ is a hyperplane in $E'$, this forces $E' = E+\ce w$. Due to $w\in\hat T_e Z \subseteq \hat T_e X$, we then conclude
    \[
        v \in E' = E+\ce w \subseteq E+\hat T_e X.
    \]
    This holds for each $e\in \tilde Y_1$, so
    \[
        v \in \bigcap_{e\in \tilde Y_1}(E+\hat T_e X),
    \]
    proving $v\in \ND_X(E)$.
\end{proof}

\begin{corollary}
    \label{cor:criterion}
    $E\subseteq \ND_X(E)$, and if equality occurs, then $E$ is $X$-unliftable.
\end{corollary}

\begin{lemma}[{\cite[Lemma 9.]{draisma}}]
    \label{lem:secant-segre-tangent}
    The regular locus of $\sigma_{r}\Seg(\pe B\times\pe C)$ consists precisely of points $[e]$ corresponding to rank $r$ matrices $e$ and the tangent cone at $[e]$ is given by \[\hat T_{[e]}\sigma_{r}\Seg(\pe B\times\pe C) = \set{M\in \Hom(B^*, C)\mid M(\ker e)\subseteq \im e}.\]
\end{lemma}

\begin{proof}[Proof of Theorem~\ref{thrm:main-liftability}]
    We take $X:= \sigma_{r} \Seg(\pe B\times\pe C)$ in Corollary~\ref{cor:criterion}, weaken the statement by only considering points of $\tilde Y_i\cap X_\reg$ in place of $\tilde Y_i$ and apply Lemma~\ref{lem:secant-segre-tangent}.
\end{proof}

\section{Degenerate rank $1$ loci}\label{sec:rank-one-degeneracies}

A simple way for the rank $1$ locus to achieve degenerate geometric rank is for the whole $\pe T_A(A^*)$ to be contained inside the Segre variety.
Another simple example is when $B$ and $C$ are two-dimensional and $\pe T_A(A^*)=\pe^2$ is the space of symmetric $2\times2$ matrices.
There, the rank $1$ locus becomes the curve of rank $1$ symmetric matrices.
Both the rank $1$ locus (curve) and the rank $2$ locus (everything) achieve
\[
    \GR(T) = 1+1 = 0+2 = 2.
\]
The following theorem shows that the above two are essentially the only ways for the rank $1$ locus to achieve degenerate geometric rank:

\begin{theorem}
    \label{thrm:pawel}
    If the rank $1$ locus of $T\in A\otimes B\otimes C$ achieves $\GR(T)$ and $T$ is of degenerate geometric rank, at least one of the following holds:
    \begin{enumerate}[label={\textup{(\alph*)}}]
        \item $\pe T_A(A^*)$ contains a linear subspace of $\Seg(\pe B\times\pe C)$.
        \item There exists a quotient $A\surjto A'$ such that $T_A((A')^*)$ is equivalent up to a $\GL(B)\times\GL(C)$ action to the space of symmetric $2\times 2$ matrices.
        Moreover, the rank $2$ locus of $T$ also achieves $\GR(T)$.
    \end{enumerate}
\end{theorem}

\begin{corollary}
    If $r=1$ is the only $r$ for which the rank $r$ locus of $T$ achieves $\GR(T)$ and $T$ is of degenerate geometric rank, then $\pe T_A(A^*)$ contains a linear subspace of $\Seg(\pe B\times\pe C)$.
\end{corollary}

We recall a useful theorem:
\begin{theorem}[Hopf]
    \label{thrm:hopf}
    Let $\phi: U \otimes V \to W$ be a linear map of $\ce$-vector spaces that is injective after restriction to subspaces of the form $s \otimes V$, $U \otimes t$ for nonzero $s$, $t$, in $U$, $V$ respectively.
    Then
    \[ \dim(\im(\phi)) \geq \dim(U) + \dim(V) -1. \]
\end{theorem}
\begin{proof}
    See \cite[Proposition 1.3]{smith}.
\end{proof}

\begin{lemma}
    \label{lem:curve}
    Let $X$ be an irreducible curve contained in $\Seg(\pe B\times\pe C)\subset \pe(B\otimes C)$, with projections
    \[
        \pr_B: \Seg(\pe B\times \pe C)\to \pe B,\qquad \pr_C: \Seg(\pe B\times \pe C)\to \pe C.
    \]
    Then $\dim\gener X\geq \dim\gener{\pr_B(X)}+\dim\gener{\pr_C(X)}$.
\end{lemma}
\begin{proof}
    Let $B'\subset B$, $C'\subset C$, $E\subset B\otimes C$ be vector subspaces such that $\pe B' = \gener{\pr_B(C)}$, $\pe C' = \gener{\pr_C(X)}$ and $\pe E=\gener X$; clearly $E\subset B'\otimes C'$.
    Let $\Hom(X,\ce)$ denote the space of regular functions $\hat X\to \ce$ and consider the function $\phi:(B')^*\otimes (C')^*\to E^*$ defined by the diagram
    \[\begin{tikzcd}[column sep = 5em]
        (B')^*\otimes (C')^* \arrow[rr, bend right, "\phi"]\arrow[r, phantom, yshift=.75em, "{\scriptstyle\text{multiplication}}"]\arrow[r, phantom, "\simeq"] & (B'\otimes C')^*\arrow[r, phantom, yshift=.75em, "{\scriptstyle\text{restriction}}"]\arrow[r] & E^*\arrow[r, phantom, yshift=.75em, "{\scriptstyle\text{restriction}}"]\arrow[r, hook, swap, "\rho"] & \Hom(\hat X,\ce).
    \end{tikzcd}\]

    Now we wish to apply Hopf's theorem on $\phi$.
    We need to verify it is injective on subspaces of the form $s\otimes (C')^*$ or $(B')^*\otimes t$ for nonzero $s\in (B')^*$ and $t\in (C')^*$.
    This is equivalent to verifying that for nonzero elements $\beta\in(B')^*$ and $\gamma\in (C')^*$, the element $\phi(\beta\otimes\gamma)$ is nonzero in $E^*$.
    It suffices to verify that $\rho(\phi(\beta\otimes\gamma))\neq0$. Taking an arbitrary point $b\otimes c\in \hat X\subset \widehat\Seg(\pe B\otimes \pe C)$, we have
    \[
        \rho(\phi(\beta\otimes\gamma))(b\otimes c) = \beta(b)\cdot\gamma(c).
    \]
    Since $\pr_B(X)$ spans $\pe B'$, the locus of points $b\otimes c\in \hat X$ with $\beta(b)=0$ is a proper closed subset for each $\beta$.
    Analogously, the same happens for $\gamma$, hence the zero set of any single $\rho(\phi(\beta\otimes\gamma))$ is a union of two proper closed subsets of $\hat X$, which is irreducible, so this union cannot be the entirety of $\hat X$.
    Thus $\rho(\phi(\beta\otimes\gamma))$ is nonzero somewhere on $\hat X$.

    Applying Hopf's theorem now gives
    \[
        \dim E = \dim E^*\geq \dim(\im(\phi)) \stackrel{\text{Hopf}}{\geq} \dim(B')^*+\dim(C')^*-1 = \dim(B')+\dim(C')-1,
    \]
    which is equivalent to $\dim\pe E \geq \dim\pe B' + \dim\pe C'$, as we wished to prove.
\end{proof}

\begin{lemma}
    \label{lem:cuttingdown}
    Suppose that the rank $1$ locus of $T\in A\otimes B\otimes C$ achieves $\GR(T)$ and this locus is at least one-dimensional.
    If $T'$ is a restriction of $T$ induced by a projection $A\surjto A'$ onto a general hyperplane, the rank $1$ locus of $T'$ also achieves $\GR(T')$.
\end{lemma}
\begin{proof}
    If suffices to show that $\codim_{\pe A^*}Y^r_{T} = \codim_{\pe (A')^*}Y^r_{T'}$ for each $r\geq1$, since all inequalities between the minimands contributing to $\GR(T)$ will be preserved upon passing to $T'$.

    A positive-dimensional irreducible projective variety $X\subset \pe^N$ intersects a general hyperplane $H\subset\pe^N$ with $\dim(X\cap H) = \dim X-1$.
    Since $Y^1_T\subseteq Y^r_T$ for all $r$ and $\dim Y^1_T\geq1$, each rank locus has a positive-dimensional irreducible component.
    Therefore, cutting with a general hyperplane $H = \pe(A')^*$ will yield
    \[
        \codim_{\pe (A')^*}Y^r_{T'} = \codim_{H}\zav{Y^r_T\cap H} = (\dim\pe A^*-1) - \zav{\dim Y^r_T-1} = \codim_{\pe A^*}Y^r_T
    \]
    as we wished to prove.
\end{proof}

Now we prove the main result.

\begin{proof}[Proof of Theorem~\ref{thrm:pawel}]
    Note that because of the degenerate geometric rank, the locus $\pe T_A(A^*)\cap \Seg(\pe B\times \pe C)$ must be positive-dimensional.
    First, we consider the case of $\dim(\pe T_A(A^*)\cap \Seg(\pe B\times \pe C))=1$.
    Let $X$ be any one-dimensional irreducible component of $\pe T_A(A^*)\cap \Seg(\pe B\times \pe C)$ and denote $\pe E:=\gener X$.
    We equip $\Seg(\pe B\times\pe C)$ with projections $\pr_B$, $\pr_C$ as before and denote $\pe B':=\gener{\pr_B(X)}$, $\pe C' := \gener{\pr_C(X)}$.
    Then Lemma~\ref{lem:curve} yields
    \begin{equation}
        \tag{$\ast$}
        \dim\pe E \geq \dim\pe B' + \dim\pe C'.
    \end{equation}

    We extract a different inequality from the rank $1$ locus achieving $\GR(T)$. We have
    \[
        \GR(T) = \codim_{\pe A^*} Y^1_T + 1 = \dim\pe A^*,
    \]
    since by assumption the rank $1$ locus is one-dimensional.
    On the other hand, $\pe E$ is contained both in $\pe T_A(A^*)$ and $\pe(B'\otimes C')$.
    Assuming $\dim B'\leq \dim C'$, $B'\otimes C'$ is a space of matrices of bounded rank $\dim B'$.
    Thus $\pe E$ is contained within the rank $\dim B'$ locus, hence
    \[
        \codim_{\pe A^*}Y^{\dim B'}_T+\dim B' = \dim\pe A^* - \dim Y^{\dim B'}_T +\dim B' \leq \dim \pe A^* -\dim\pe E +\dim B'.
    \]
    Since the left hand side is an upper bound on $\GR(T)$, we conclude
    \begin{align}
        \nonumber
        \dim \pe A^* &\leq \dim\pe A^* - \dim\pe E + \dim \pe B' + 1,\\
        \tag{$\dag$}
        \dim \pe E &\leq \dim\pe B'+1.
    \end{align}
    Putting ($\ast$) and ($\dag$) together, we extract
    \[
        \dim\pe B'+1 \geq \dim\pe B'+\dim\pe C',
    \]
    i.e. $1\geq\dim\pe C'\geq\dim\pe B'$.
    Furthermore, we must have $\dim\pe(B'\otimes C')\geq\dim X = 1$, which implies at least one of $\pe B'$, $\pe C'$ is one-dimensional.
    Thus $\dim\pe C' = 1$ and $\dim\pe B'\in\set{0,1}$.

    If $\dim\pe B' = 0$, then both sides of the inclusion $X\subset\pe (B'\otimes C')$ are one-dimensional irreducible subvarieties of $\pe(B\otimes C)$, so we see $X=\pe(B'\otimes C')$.
    Hence $\pe T_A(A^*)$ indeed contains a linear subspace of the Segre variety, as posited by (a).

    If $\dim\pe B' = 1$, we assume $\pe E$ does not contain a subspace of the Segre variety and prove that it is, up to the action of $\GL(B)\times\GL(C)$, the space of $2\times2$ matrices, as posited by (b).
    First, we observe that the restrictions of $\pr_B$, $\pr_C$ to $X$ are injective: if $\pr_C$ restricted to $X$ is not injective, we can find two distinct $[b_1\otimes c], [b_2\otimes c]\in X$.
    Then $\pe E$ contains $\gener{[b_1\otimes c], [b_2\otimes c]} = \Seg(\gener{[b_1], [b_2]}\times [c])$, which is a linear subspace of the Segre variety, a~contradiction.

    Since restrictions $\pr_B\vert_X$, $\pr_C\vert_X$ are injective projective homomorphisms and $\dim X = \dim \pe B' = \dim\pe C' = 1$, they must be surjective, so $X$ constitutes the graph of an isomorphism $\pe^1 = \pe B' \simeq \pe C' = \pe^1$.
    Such an isomorphism is given by an invertible $2\times 2$ matrix modulo scalars (see e.g. \cite[15.5.A]{rising-sea}).
    By considering appropriate bases of $B'$ and $C'$, we may assume the isomorphism is given by the identity.
    Then, $X$ is the curve of rank $1$ symmetric matrices inside $\Seg(\pe^1\times\pe^1)$ and its span $\pe E$ becomes the subspace of all symmetric $2\times 2$ matrices, which is what we wished to prove.

    Moreover, since we reached an equality in ($\dag$) and we have $\dim B'=2$, equality must be occurring in
    \[
        \codim_{\pe A^*}(Y^2_T)+2 \geq \codim_{\pe A^*}Y^1_T + 1,
    \]
    which means the rank $2$ locus also achieves $\GR(T)$.

    \bigskip
    Now, we prove the general case by induction on $\dim(\pe T_A(A^*)\cap \Seg(\pe B\times\pe C))$.
    We choose a general hyperplane $H\subset B\otimes C$ and the restriction $T'$ of $T$ induced $A\surjto A'$, where $(A')^* = T_A^{-1}(T_A(A^*)\cap H)$. 
    By Lemma~\ref{lem:cuttingdown}, the rank $1$ locus of $T'$ still achieves $\GR(T')$,
    while the dimension of $\pe T'_A((A')^*) \cap \Seg(\pe B\times \pe C)$ decreases by one.
    
    By the inductive hypothesis, the theorem holds for $T'$.
    Therefore, there is a linear subspace of $\Seg(\pe B\times \pe C)$ or a subspace equivalent to $2\times 2$ matrices contained in $\pe T'_A((A')^*) = \pe T_A((A')^*)$, so it is also contained in the larger space $\pe T_A(A^*)$.
    Lastly, we know that by cutting with a general hyperplane, the dimensions of all rank loci decreased by the same amount. Hence if the conclusion (b) occurred for $T'$, meaning that the rank $2$ locus achieved $\GR(T')$, the same must happen for $T$.
\end{proof}

\section{Examples of nonlinearly degenerate geometric rank}
\label{sec:examples}

One of our goals has been to find examples of tensors with degenerate geometric rank where the degeneracy is achieved by some locus that is not just a linear space of matrices of bounded rank. We provide examples we know.

\subsection{Matrix multiplication}

\begin{example}[$M_{\gener n}$ is $rn$-unliftable]
    \label{ex:matmult}
    Consider the space of matrices $E$ corresponding to the multiplication of $n\times n$ matrices. This consists of matrices of the form
    \[
        e_A := \mtrx{A&&&\\&A&&\\&&\ddots&\\&&&A} \in \Hom(\ce^{n^2}, \ce^{n^2}) \text{ for $A\in\Hom(\ce^n, \ce^n)$.}
    \]
    Clearly, $\rank(e_A) = n \rank(A)$, so only ranks $rn$ for $r=\rank(A)\in\set{0,\dots,n}$ occur.

    The locus $Y^{rn}:=Y^{rn}_{M_{\gener n}}$ of matrices of rank at most $rn$ in $\pe E$ corresponds to $\rank(A)\leq r$.
    Thus, it is a copy of $\sigma_{r}\Seg(\pe^{n-1}\times\pe^{n-1})$ embedded linearly in $\pe(\Hom(\ce^{n^2}, \ce^{n^2}))$.
    The locus is irreducible and has codimension $(n-r)^2$ in $\pe\Hom(\ce^n,\ce^n)\simeq \pe E$, so it contributes $(n-r)^2+rn = n^2-rn+r^2$ to the minimum that defines $\GR(M_{\gener n})$.
    This attains its minimum $\GR(M_{\gener n})=\ceil{\frac34n^2}$ at $r=\floor{\frac n2}$ and $r=\ceil{\frac n2}$ (cf. \cite[Theorem 6.1]{kopparty-moshkovitz-zuiddam}).

    Let us show that for every $0<r<n$, the space $E$ is $rn$-unliftable.
    Since $Y^{rn}$ is irreducible and its regular locus consists of $e_A$ for $\rank(A)=r$, we just need to exhibit a collection of $e_A$'s such that their corresponding $E+\hat T_{[e_A]}\sigma_{rn}\Seg(\pe^{n^2-1}\times\pe^{n^2-1})$ intersect to $E$.
    We pick all possible matrices $A$ that contain $r$ entries $1$, at most one in each in row and column, and zeros everywhere else.
    To see we reach $E$ as the intersection, view $n^2\times n^2$ matrices written using $n\times n$ blocks of size $n\times n$.
    Then $\hat T_{[e_A]}\sigma_{rn}\Seg(\pe^{n^2-1}\times\pe^{n^2-1})$ allows arbitrary entries in any row or column with a one, whereas $E$ will force the remaining entries to either be zero if they are in an off-diagonal block, or to agree across different diagonal blocks.
    Taking the conjunction of such conditions $E+\hat T_{[e_A]}\sigma_{rn}\Seg(\pe^{n^2-1}\times\pe^{n^2}-1)$ over all choices of $A$, we see that off-diagonal blocks must be zero and the diagonal blocks must be the same, which is exactly the description of $E$.

    We illustrate with an example for $n=2$, $r=1$: our chosen matrices $e_A$ will be
    \begin{align*}
        e_1 &= \mtrx{1&0&&\\0&0&&\\&&1&0\\&&0&0}, &
        e_2 &= \mtrx{0&1&&\\0&0&&\\&&0&1\\&&0&0}, \\
        e_3 &= \mtrx{0&0&&\\1&0&&\\&&0&0\\&&1&0},&
        e_4 &= \mtrx{0&0&&\\0&1&&\\&&0&0\\&&0&1}.
    \end{align*}
    We compute
    \begin{align*}
        E+\hat T_{[e_1]}\sigma_2\Seg(\pe^3\times\pe^3) &= \mtrx{x&y&&\\z&w&&\\&&x&y\\&&z&w} + \mtrx{*&*&*&*\\ *&0&*&0\\ *&*&*&*\\ *&0&*&0} = \mtrx{*&*&*&*\\ *&w&*&0\\ *&*&*&*\\ *&0&*&w}
    \end{align*}
    and analogously
    \begin{align*}
        E+\hat T_{[e_2]}\sigma_2\Seg(\pe^3\times\pe^3) &= \mtrx{
            *&*&*&*\\
            z&*&0&*\\
            *&*&*&*\\
            0&*&z&*
        },\\
        E+\hat T_{[e_3]}\sigma_2\Seg(\pe^3\times\pe^3) &= \mtrx{
            *&y&*&0\\
            *&*&*&*\\
            *&0&*&y\\
            *&*&*&*
        },\\
        E+\hat T_{[e_4]}\sigma_2\Seg(\pe^3\times\pe^3) &= \mtrx{
            x&*&0&*\\
            *&*&*&*\\
            0&*&x&*\\
            *&*&*&*
        }.
    \end{align*}
    Then, intersecting gives
    \[
    \RND_2(E) \subseteq
    \mtrx{x&y&&\\z&w&&\\&&x&y\\&&z&w} = E.
    \]
\end{example}

\subsection{Structure tensors}

\begin{example}[octonions]
Let $T_{\oct}\in \ce^8\otimes\ce^8\otimes\ce^8$ be the structure tensor of the complexified octonions $\oct$.
Explicitly, the corresponding space of matrices may be written as
\[
    \mtrx{
        x_0 &  x_1 &  x_2 &  x_3 &  x_4 &  x_5 &  x_6 &  x_7\\
        x_1 & -x_0 &  x_3 & -x_2 &  x_5 & -x_4 & -x_7 &  x_6\\
        x_2 & -x_3 & -x_0 & -x_1 &  x_6 &  x_7 & -x_4 & -x_5\\
        x_3 &  x_2 & -x_1 & -x_0 &  x_7 & -x_6 &  x_5 & -x_4\\
        x_4 & -x_5 & -x_6 & -x_7 & -x_0 &  x_1 &  x_2 &  x_3\\
        x_5 &  x_4 & -x_7 &  x_6 & -x_1 & -x_0 & -x_3 &  x_2\\
        x_6 &  x_7 &  x_4 & -x_5 & -x_2 &  x_3 & -x_0 & -x_1\\
        x_7 & -x_6 &  x_5 &  x_4 & -x_3 & -x_2 &  x_1 & -x_0
    }.
\]
A nonzero element of this space either has rank $4$ if $\sum x_i^2 = 0$, or $8$ otherwise.
In particular, it means the rank $4$ locus is a hypersurface, which results in $\GR(T_{\oct}) = 1+4 = 5$.
Note that the rank $4$ locus has a neat algebraic interpretation: the equation defining the hypersurface is the octonion \emph{norm}, which is multiplicative, despite the octonions being nonassociative.

Computationally, we have verified that the rank $4$ locus is unliftable.
\end{example}

The octonions can be viewed as a member of the larger family of Cayley-Dickson algebras, produced inductively by using an $n$-dimensional algebra to construct a $2n$-dimensional algebra.
In the classical setting over the reals, the sequence starts with real numbers, complex numbers, quaternions, octonions, and so on.
The $2\times2$ matrix multiplication tensor comes up here, because when passing to the complex numbers, the quaternions $\mathbb H\otimes_{\er}\ce$ are isomorphic to the algebra of $2\times 2$ matrices.
Unfortunately though, we have not observed degenerate geometric rank with higher Cayley-Dickson algebras. We note that starting with the $16$-dimensional \emph{sedenions}, the norm function given by a sum of squares stops being multiplicative.

\subsection{Minimal border rank}

Tensors of minimal border may be good source of examples of degenerate geometric rank.
In $\ce^m\otimes\ce^m\otimes\ce^m$ for $m\leq 5$, these have been classified in \cite[Section 4]{jagiella-jelisiejew}.
Going through the list and adopting their notation, we see nonlinear degenerate rank loci for the following eight tensors; for each, we list the loci for all nontrivial ranks which occur in the space.
\begin{itemize}
    \item The tensor
    \[
        T_{1,10} = \mtrx{
            x_0 & 0 & 0 & 0 & 0 \\
            0 & x_0 & 0 & 0 & 0 \\
            0 & 0 & x_0 & 0 & 0 \\
            0 & x_1 & x_3 & x_0 & 0 \\
            0 & x_2 & x_4 & 0 & x_0 \\
        }
    \]
    has a rank $2$ locus given by the ideal $(x_0)$ and a rank $1$ locus given by $(x_0, x_2x_3-x_1x_4)$.
    The geometric rank is $3$ and it is achieved by both of the above loci.

    \item The tensor
    \[
        T_{1,11} = \mtrx{
            x_0 & 0 & 0 & 0 & 0 \\
            0 & x_0 & 0 & 0 & 0 \\
            0 & 0 & x_0 & 0 & 0 \\
            x_2 & x_3 & x_4 & x_0 & 0 \\
            x_1 & x_2 & x_3 & 0 & x_0
        }
    \]
    has a rank $2$ locus given by $(x_0)$ and a rank $1$ locus given by $(x_0, x_3^2-x_2x_4, x_2x_3-x_1x_4, x_2^2-x_1x_3)$.
    The geometric rank is $3$ and it is achieved solely by the linear rank $2$ locus, with the rank $1$ locus only contributing $4$ to the geometric rank.

    \item The tensor
    \[
        T_{1,12} = \mtrx{
            x_0 & 0 & 0 & 0 & 0 \\
            0 & x_0 & 0 & 0 & 0 \\
            0 & 0 & x_0 & 0 & 0 \\
            x_1 & 0 & x_4 & x_0 & 0 \\
            x_2 & x_3 & x_1 & 0 & x_0
        }
    \]
    has a rank $2$ locus given by $(x_0)$ and a rank $1$ locus given by $(x_0, x_3x_4, x_1x_3, x_1^2-x_2x_4)$.
    This rank $1$ locus thus has two components $(x_0,x_3, x_1^2-x_2x_4)$ and $(x_0, x_1,x_4)$.
    The geometric rank is $3$ and it is achieved solely by the linear rank $2$ locus, with the rank $1$ locus only contributing $4$ to the geometric rank.

    \item The tensor
    \[
        T_{1,19} = \mtrx{
            x_0 & 0 & 0 & 0 & 0 \\
            0 & x_0 & 0 & 0 & 0 \\
            x_1 & 0 & x_0 & 0 & 0 \\
            x_3 & x_4 & 0 & x_0 & 0 \\
            x_2 & x_3 & x_1 & 0 & x_0
        }
    \]
    has a rank $3$ locus given by $(x_0)$, a rank $2$ locus given by $(x_0, x_1x_4)$ and a rank $1$ locus given by $(x_0,x_1,x_3^2-x_2x_4)$.
    The geometric rank is $4$, achieved by all three of the above loci.

    \item The tensor
    \[
        T_{2,7} = \mtrx{
            x_0 & 0 & 0 & 0 & 0 \\
            0 & x_0 & 0 & 0 & 0 \\
            x_1 & x_2 & x_0 & 0 & 0 \\
            x_3 & -x_1 & 0 & x_0 & 0 \\
            0 & 0 & 0 & 0 & x_0 + x_4
        }
    \]
    has a rank $3$ locus given by $(x_0)$, a rank $2$ locus given by $(x_0, x_1^2x_4+x_2x_3x_4)$ and a rank $1$ locus given by $(x_0, x_3x_4,x_2x_4, x_1x_4, x_1^2+x_2x_3)$.
    The rank $2$ locus has two components $(x_0,x_4)$ and $(x_0,x_1^2+x_2x_3)$, while the rank $1$ locus has two components $(x_0,x_4,x_1^2+x_2x_3)$ and $(x_0,x_1,x_2,x_3)$.
    The geometric rank is $4$, achieved by all three of the above loci.

    \item The tensor
    \[
        T_{\O_{58}} = \mtrx{
            x_0 & 0 & x_1 & x_2 & x_4 \\
            x_4 & x_0 & x_3 & -x_1 & 0 \\
            0 & 0 & x_0 & 0 & 0 \\
            0 & 0 & - x_4 & x_0 & 0 \\
            0 & 0 & 0 & x_4 & 0
        }.
    \]
    as a space of matrices is of bounded rank $4$.
    It further has a rank $2$ locus $(x_0,x_4)$ and a rank $1$ locus $(x_0,x_4,x_1^2+x_2x_3)$.
    The geometric rank is $4$, achieved by all three of the above loci.

    \item The tensor
    \[
        T_{\O_{57}} = \mtrx{
            x_0 & 0 & x_1 & x_2 & x_4 \\
            0 & x_0 & x_3 & -x_1 & 0 \\
            0 & 0 & x_0 & 0 & 0 \\
            0 & 0 & 0 & x_0 & 0 \\
            0 & 0 & 0 & x_4 & 0
        }
    \]
    as a space of matrices is of bounded rank $4$.
    It further has a rank $3$ locus given by $(x_0)$, a rank $2$ locus given by $(x_0,x_3x_4)$ and a rank $1$ locus given by $(x_0,x_4,x_1^2+x_2x_3)$.
    The geometric rank is $4$, achieved by all four of the above loci.

    \item The tensor
    \[
        U_{2,7} = \mtrx{
            x_0 & 0 & 0 & 0 \\
            0 & x_0 & 0  & 0 \\
            x_1 & x_2 & x_0 & 0 \\
            x_3 & -x_1 & 0 & x_0
        }
    \]
    has a rank $2$ locus $(x_0)$ and a rank $1$ locus $(x_0, x_1^2+x_2x_3)$.
    The geometric rank is $3$, achieved by both of the above loci.
\end{itemize}

Overall, we find no examples of degenerate geometric rank being achieved solely by a nonlinear locus, as there is always a linear locus which achieves the geometric rank jointly with the nonlinear locus.

We also note that the examples are consistent with the result of Section~\ref{sec:rank-one-degeneracies}: in each case when the rank $1$ locus achieves the geometric rank, it either contains a linear space, or it contains a copy of $\pe S^2\ce^2 \cap \Seg(\pe^1\times\pe^1)$, in which case the rank $2$ locus also achieves the geometric rank, as Theorem~\ref{thrm:pawel} predicts.

\subsection{Representation-theoretic constructions}

We give an example coming from the representation theory of $\SL_3$.

\begin{example}
The following tensor in $\ce^6\otimes\ce^6\otimes\ce^6$ is the unique up to scale $\SL_3$-invariant tensor in $S_2\ce^3$:
\[
    E_6 := \mtrx{
            0  &   0    & 0 &  2x_5 & 2x_4 & 2x_3 \\
            0 & 2x_5   & x_4 &    0 &  -x_2 &  -x_1 \\
            0 &  -x_4 & -2x_3 &  -x_2 &  -x_1 &    0 \\
        2x_5  &   0 & -2x_2   &  0  &   0 & 2x_0 \\
        -x_4  & -x_2  &  x_1  &   0 & 2x_0 &    0 \\
        2x_3 & 2x_1  &   0 & 2x_0 &    0  &   0
    }.
\]
The tensor corresponds to the $\SL_3$-equivariant map $S_2\ce^3\to \Hom(S_2\ce^3, S_2\ce^3)$.
In basis, we can construct the map by regarding $S_2\ce^3$ as quadratic forms in variables $y_1,y_2,y_3$, with the description above using the correspondence $(x_0,\dots,x_5)\approx(y_1^2, y_1y_2, y_1y_3,y_2^2,y_2y_3,y_3^2)$.
Then, for a monomial $y_ay_b\in S_2\ce^3$, we compute what the monomial $y_iy_j$ maps to under the map $\phi\in\Hom(S_2\ce^3, S_2\ce^3)$ induced by $y_ay_b$.
We consider the two possible pairings $\set{a,i},\set{b,j}$ and $\set{a,j},\set{b,i}$, each contributing a possibly zero term to $\phi(y_iy_j)$.
For the first pairing, if $a=i$ or $b=j$, the pairing contributes $0$.
Otherwise, $\set{a,i}$ and $\set{b,j}$ are both two-element subsets of $\set{1,2,3}$, so let
\[
    \set{s} := \set{1,2,3}\setminus\set{a,i},\qquad \set{t}:=\set{1,2,3}\setminus\set{b,j},
\]
then the pairing contributes the term
\[
    -\sgn((ais))\sgn((bjt))y_sy_t
\]
to the sum, where $\sgn$ denotes the sign of a permutation.
The pairing $\set{a,j},\set{b,i}$ also contributes terms analogously.

As an example, we compute the image of $y_1y_3$ under the map induced by $y_1y_2$:
\[
    \zav{\text{contribution of $\set{1,1},\set{2,3}$}} + \zav{\text{contribution of $\set{1,3},\set{2,1}$}} = 0 + \zav{-\sgn(132)\sgn(213)y_2y_3} = -y_2y_3,
\]
justifying the entry $-x_1$ in the third row and fifth column.

Similarly, the image of $y_3^2$ under the map induced by $y_1y_2$ is:
\[
    2\cdot\zav{\text{contribution of $\set{1,3},\set{2,3}$}} = 2\cdot (-\sgn((132))\sgn((231))y_2y_1) = 2y_1y_2,
\]
justifying the entry $2x_1$ in the sixth row and second column.

As a space of matrices, $E_6$ has a rank $4$ locus given by the degree $3$ hypersurface $(x_2^2x_3 - x_1x_2x_4 + x_0x_4^2 + x_1^2x_5 - 4x_0x_3x_5)$, which is actually the second secant variety of the degree two Veronese embedding of $\pe^2$; and a rank $3$ locus given by
\[
    (x_4^2 - 4x_3x_5,\ x_2x_4 - 2x_1x_5,\ 2x_2x_3 - x_1x_4,\ x_2^2 - 4x_0x_5,\ x_1x_2 - 2x_0x_4,\ x_1^2 - 4x_0x_3),
\]
the degree two Veronese embedding of $\pe^2$.
Of these, only the rank $4$ locus achieves $\GR(E_6) = 1+4 = 5$.

By a computer calculation of the liftability criterion, we have verified that $E_6$ is $4$-unliftable in all flattenings.
\end{example}

\section*{Acknowledgments}

This project began at the Workshop on the Applications of Commutative Algebra at the Fields Institute in Toronto in May 2025. We are grateful to its organizers, to J.M. Landsberg and Claudiu Raicu as leaders of the focus group on Tensors and to all of our fellow group members.
We also thank Amichai Lampert and Mateusz Michałek for helpful comments.

This work has been supported by European Union’s HORIZON–MSCA-2022-DN-JD programme under the Horizon Europe (HORIZON) Marie Skłodowska-Curie Actions, grant agreement 101120296 (TENORS).

\bibliographystyle{alpha}
\bibliography{main}
\end{document}